\documentclass{amsart}
\usepackage[utf8]{inputenc}
\usepackage{amsfonts,latexsym,amssymb,amsmath,amsthm,color}
\usepackage{enumerate}
\usepackage{captdef}
\usepackage[dvips]{graphicx} 
\usepackage{enumerate,graphicx,graphpap}
\usepackage[matrix]{xy}
\usepackage{cases}
\usepackage{mathrsfs}
\input xy
\xyoption{all} \xyoption{poly}

\theoremstyle{plain}
\newtheorem{teor}{Theorem}[section]
\newtheorem{lema}[teor]{Lemma}

\newtheorem{prop}[teor]{Proposition}

\theoremstyle{definition}
\newtheorem{defi}{Definition}[section]

\newtheorem{nota}[teor]{Remark}

\newtheoremstyle{teoremacita}
{3pt}
{3pt}
{\itshape}
{}
{\bfseries}
{}
{ }
{\thmname{#1}\thmnumber{ #2'}\thmnote{ #3}.}
\theoremstyle{teoremacita} \newtheorem*{teor*}{}

\newcommand{\be}{\begin{enumerate}}
\newcommand{\ee}{\end{enumerate}}

\newcommand{\bi}{\begin{itemize}}
\newcommand{\ei}{\end{itemize}}

\def\N{\mathbb N}

\def\R{\mathbb R}
\def\C{\mathbb C}

\def\a{\alpha}

\def\e{\varepsilon}

\def\d{\partial}

\begin{document}

\title{An extension of Borel-Laplace methods and monomial summability}

\author{Sergio A. Carrillo}
\address[Sergio A. Carrillo]{Escuela de Ciencias Exactas e Ingenier\'{i}a \\ Universidad Sergio Arboleda \\ Calle 74,  \# 14-14 \\Bogot\'{a} - Colombia}
\email{sergio.carrillo@correo.usa.edu.co, sergio\_carrillo30@hotmail.com}

\author{Jorge Mozo-Fern\'{a}ndez}
\address[Jorge Mozo Fern\'{a}ndez]{Dpto. \'{A}lgebra, An\'{a}lisis Matem\'{a}tico, Geometr\'{\i}a y Topolog\'{\i}a \\
Facultad de Ciencias, Universidad de Valladolid \\
Campus Miguel Delibes\\
Paseo de Bel\'{e}n, 7\\
47011 Valladolid - Spain}
\email{jmozo@maf.uva.es}

\thanks{First author was supported by an FPI grant (call 2012/2013) conceded by Universidad de Valladolid, Spain.\\ Both authors partially supported by the Ministerio de Econom\'{\i}a y Competitividad from Spain, under the Project ``\'{A}lgebra y Geometr\'{\i}a en Din\'{a}mica Real y Compleja III" (Ref.: MTM2013-46337-C2-1-P)}


\date{\today}

\begin{abstract}
In this paper we will show that monomial summability can be characterized using Borel-Laplace like integral transformations depending of a parameter $0<s<1$. We will apply this result to prove 1-summability in a monomial of formal solutions of a family of partial differential equations.
\end{abstract}

\maketitle

\section{Introduction}

Monomial summability was developed \cite{Monomial summ} in order to understand summability properties of formal solutions of a class of singularly perturbed differential systems (doubly singular differential systems). In contrast with the classical theory of summability, there was not available an approach to the concept using integral transformation, i.e. a Borel-Laplace method of summation. On the other hand, in \cite{BalserMozo} these  formal solutions were studied in the linear case using certain integral transformations, also introduced later in \cite{Balser3} for any number of variables, but the relation between the methods of summation was not clear. It turns out that using an adaptation of those operators the problem of understanding that relation can be solved. In fact, one of the main results in this work is the following one (appropriate definitions will be introduced later):

\begin{teor*}[Theorem \ref{Monomial summability t Borel Laplace}]Let $\hat{f}$ be a $1/k-$Gevrey series in the monomial $x_1^px_2^q$. Then it is equivalent:
\begin{enumerate}
\item $\hat{f}\in E\{x_1,x_2\}_{1/k,d}^{(p,q)}$, i.e. $\hat{f}$ is $k-$summable in $x^{p}_{1}x_{2}^{q}$ in direction $d$.

\item There is $0<s<1$ such that $\hat{f}$ is $k-(s,1-s)-$Borel summable in the monomial $x_1^px_2^q$ in direction $d$.

\item For all $0<s<1$, $\hat{f}$ is $k-(s,1-s)-$Borel summable in the monomial $x_1^px_2^q$ in direction $d$.
\end{enumerate}
In all cases the corresponding sums coincide.
\end{teor*}

This result shows in particular that the integral methods introduced in \cite{BalserMozo} agree with the notion of monomial summability. And, moreover, they can be used in order to treat more complicated systems, for instance, of partial differential equations, namely, families of PDEs as follows:

$$x_1^px_2^q\left(\frac{s}{p}x_1\frac{\d \textbf{y}}{\d x_1}+\frac{1-s}{q}x_2\frac{\d \textbf{y}}{\d x_2}\right)=C(x_1,x_2)\textbf{y}+\gamma(x_1,x_2).$$

These systems generalize the ones in \cite{Monomial summ} for $s=0$ and $s=1$, under the assumption that $C(0,0)$ is invertible.

The plan of the paper is as follows: in Section \ref{Sumabilidad}, the notions of asymptotic expansions and summability, both in the one variable, and in the monomial case, are reviewed. Monomial Borel and Laplace transforms are defined and investigated in Section \ref{BorelyLaplace}, in order to introduce the summation methods using these operator in Section  \ref{sumacion}. The main result of this work, Theorem \ref{Monomial summability t Borel Laplace}, is stated and proved here. Finally, Section \ref{ejemplos} provides examples of application to a family of partial differential equations. A possible further development, in order to define multisummability in the monomial case, is sketched in Section \ref{further}.

\subsection*{Acknowledgments} The second author wants to thank the Universidad Sergio Arboleda (Bogotá, Colombia), for the hospitality while preparing this paper. Both authors would like to thank Javier Rib\'{o}n, from Universidade Federal Fluminense de Niter\'{o}i (Brazil), for fruitful conversations.

\section{Summability} \label{Sumabilidad}

We recall briefly the basic notions of asymptotic expansions, Gevrey asymptotic expansions and summability in the case of one variable and their extensions to the monomial case in two variables, as introduced in the paper \cite{Monomial summ}.

Let $(E,\|\cdot\|)$ be a complex Banach space and $\hat{f}=\sum a_n x^n\in E[[x]]$. We will denote by $\mathcal{C}(U, E)$ (resp. $\mathcal{O}(U, E)$, $\mathcal{O}_b(U, E)$) the space of continuous $E-$valued maps (resp. holomorphic, holomorphic and bounded $E-$valued maps) defined on an open set $U\subset \C^l$. If $E=\C$ we will simply write $\mathcal{O}(U)$. We also denote by $\N$ the set of natural numbers including $0$, by $\N_{>0}=\N\setminus\{0\}$, by $D_r\subset\C$ the disk centered at the origin with radius $r$ and by $V=V(a,b,r)=\{x\in\C | 0<|x|<r, a<\text{arg}(x)<b\}$ an open sector in $\C$. When we want to emphasize the bisecting direction $d=(b+a)/2$ and the opening of the sector we will write $V=S(d,b-a,r)=S$. In the case $r=+\infty$ we will simply write $V(a,b)=S(d,b-a)$. Consider $f\in\mathcal{O}(V, E)$. The map $f$ is said to have $\hat{f}$ as \textit{asymptotic expansion }at the origin on $V$ (denoted by $f\sim \hat{f}$ on $V$) if for each of its proper subsectors $V'=V(a',b',r')$ ($a<a'<b'<b$, $0<r'<r$) and each $N\in\N$, there exists $C_N(V')>0$ such that \begin{equation}\label{def asym classic}
\left\|f(x)-\sum_{n=0}^{N-1} a_n x^n\right\|\leq C_N(V')|x|^N, \hspace{0.3cm} \text{ on } V'.
\end{equation}

If we can take $C_N(V')=C(V')A(V')^N N!^s$, $C(V'), A(V')$ independent of $N$, the asymptotic expansion is said to be of \textit{$s-$Gevrey type} (denoted by $f\sim_s \hat{f}$ on $V$). In this case $\hat{f}\in E[[x]]_s$, where $E[[x]]_s$ denotes the space of $s-$Gevrey series, i.e. there exist $C,A>0$ such that $\|a_n\|\leq CA^nn!^s$, for all $n\in\N$.

Asymptotic expansions are unique, and respect algebraic operations and differentiation. In the Gevrey case Watson's lemma states that if $f\sim_s 0$ on $V(a,b,r)$ and $b-a> s\pi$ then $f\equiv 0$. This is the key point to define $k-$summability, as is explained below.

\begin{defi}Let  $\hat{f}\in E[[x]]$ be a formal power series, let $k>0$ and let $d$ be a direction.
\begin{enumerate}
\item The formal series $\hat{f}$ is called \textit{$k-$summable on $V=V(a,b,r)$} if $b-a>\pi/k$ and there exists a map $f\in\mathcal{O}(V,E)$ such that $f\sim_{1/k} \hat{f}$ on $V$. If $d$ is the bisecting direction of $V$ then $\hat{f}$ is called \textit{$k-$summable in the direction $d$} and $f$ is called \textit{the $k-$sum} of $\hat{f}$ in direction $d$ (or on $V$).

\item The formal series $\hat{f}$ is called \textit{$k-$summable}, if it is $k-$summable in every direction with finitely many exceptions mod. $2\pi$ (the singular directions).
\end{enumerate}

The space of $k-$summable series in the direction $d$ will be denoted by $E\{x\}_{1/k,d}$ and the space of $k-$summable series will be denoted by $E\{x\}_{1/k}$. Both spaces inherit naturally a structure of algebra when $E$ is a Banach algebra.
\end{defi}

To calculate the $k-$sum of a $k-$summable series in a direction $d$ we have at our disposal the \textit{Borel-Laplace method} that uses the following integral transformations:

\begin{enumerate}
\item \textit{The $k-$Borel transform}, defined as $\mathcal{B}_kf(\xi)=\frac{k}{2\pi i}\int_{\gamma} f(x)e^{(\xi/x)^k} \frac{dx}{x^{k+1}},$ where $f\in\mathcal{O}_b(S,E)$, $S=S(d,\pi/k+2\epsilon, R_0)$, $0<2\epsilon<\pi/k$ and $\gamma$ is the boundary of a sector at the origin of opening greater than $\pi/k$ contained in $S$, oriented in the positive sense. It induces the \textit{formal $k-$Borel transform} $\widehat{\mathcal{B}}_k$, defined for elements of $x^k E[[x]]$ term by term using the formula $\mathcal{B}_k(x^\lambda)(\xi)=\frac{\xi^{\lambda-k}}{\Gamma(\lambda/k)}$ for $\lambda\in\C$.

\item \textit{The $k-$Laplace transform in direction $d$}, defined as $\mathcal{L}_{k,d}(g)(x)=\int_0^{e^{id}\infty} g(\xi)e^{-(\xi/x)^k}d\xi^k,$ where
$g$ is continuous on the half-line with vertex at $0$ and direction $d$ and having exponential growth of order at most $k$ on its domain. When $g$ is defined on a sector of opening less than $\pi/k$, moving $d$ leads to analytic continuation of $\mathcal{L}_{k,d}(g)$, and we will simply write $\mathcal{L}_{k}(g)$.
\end{enumerate}

In order to sum $\hat{f}\in E[[x]]_{1/k}$, consider the convergent power series $\widehat{\mathcal{B}}_k(x^k\hat{f})$ and attempt to make analytic continuation, say $\varphi$, to an infinite sector $W$ containing $d$. If this is possible and $\varphi$ has exponential growth of order at most $k$ on $W$ then $\hat{f}$ is said to be \textit{k-Borel summable in direction $d$} and its sum is defined by $f(x)=\frac{1}{x^k}\mathcal{L}_{k}(\varphi)(x)$. Since the previous integral transformations are inverses of each other and behave well under Gevrey asymptotic expansions (see e.g. \cite{Balser2}) it follows that:

\begin{teor}[Ramis]\label{Borel sum iff Ramis sumable} A power series $\hat{f}\in E[[x]]$ is $k-$Borel summable in a direction $d$ if and
only if it is $k-$summable in the direction $d$ and both sums coincide.
\end{teor}

We will extend now the concept of asymptotic expansions and $k-$summability in order to introduce  monomial asymptotics and summability in two variables is also possible. Indeed let $p,q\in\N_{>0}$ be fixed and consider the monomial $x_1^px_2^q$. For formal power series we may write any $\hat{f}=\sum a_{n,m} x_1^nx_2^m\in E[[x_1,x_2]]$ uniquely as \begin{equation}\label{para definir Tpq}
\hat{f}=\sum_{n=0}^\infty f_n(x_1,x_2)(x_1^px_2^q)^n,\hspace{0.4cm} f_n=\sum_{m<p\text{ or }j<q}a_{np+m,nq+j}x_1^mx_2^j.
\end{equation}

To ensure that each $f_n$ gives rise to an holomorphic map defined in a common polydisc at the origin it is necessary and sufficient that $\hat{f}\in \mathcal{C}=\bigcup_{r>0} \mathcal{C}_r$, where $\mathcal{C}_r=\mathcal{O}_b(D_r,E)[[x_1]]\cap \mathcal{O}_b(D_r,E)[[x_2]]$. If this is the case then $f_n\in \mathcal{E}^{(p,q)}=\bigcup_{r>0}\mathcal{E}^{(p,q)}_r$, where $\mathcal{E}^{(p,q)}_r$ is the space of holomorphic maps $g\in\mathcal{O}_b(D_r^2,E)$ with $\frac{\d^{n+m}g}{\d x_1^n\d x_2^m}(0,0)=0$ for $n\geq p$ and $m\geq q$. Each $\mathcal{E}^{(p,q)}_r$ becomes a Banach space with the supremum norm $\|g\|_r=\sup_{|x_1|,|x_2|\leq r} \|g(x_1,x_2)\|$.

Define the map $\hat{T}_{p,q}:\mathcal{C}\rightarrow \mathcal{E}^{(p,q)}[[t]]$ by $\hat{T}_{p,q}(\hat{f})=\sum_{n=0}^\infty f_n t^n$, using the decomposition (\ref{para definir Tpq}). We recall that $\hat{f}$ is said to be \textit{$s-$Gevrey in the monomial $x_1^px_2^q$} if for some $r>0$, $\hat{T}_{p,q}(\hat{f})\in\mathcal{E}_r^{(p,q)}[[t]]$ and it is a $s-$Gevrey series in $t$. The set of $s-$Gevrey series in the monomial $x_1^px_2^q$ will be denoted by $E[[x_1,x_2]]_s^{(p,q)}$.

It follows from Cauchy's estimates that $\sum a_{n,m}x_1^nx_2^m\in E[[x_1,x_2]]_s^{(p,q)}$ if and only if there exist constants $C,A>0$ such that $\|a_{n,m}\|\leq CA^{n+m}\min\{n!^{s/p}, m!^{s/q}\}$, for all $n,m\in\N$. In particular $E[[x_1,x_2]]_{Ms}^{(Mp,Mq)}=E[[x_1,x_2]]_s^{(p,q)}$ for any $M\in\N_{>0}$ and $\hat{f}\in E\{x_1,x_2\}$ if and only if $\hat{T}_{p,q}(\hat{f})\in \mathcal{E}^{(p,q)}_r\{t\}$, for some $r>0$ by taking $s=0$.

Let $E[[x_1,x_2]]_{(s_1,s_2)}$ denote the algebra of $(s_1,s_2)-$Gevrey series, i.e. $\sum a_{n,m} x_1^nx_2^m$ is $(s_1,s_2)$-Gevrey if there exist constants $C,A>0$ such that $\|a_{n,m}\|\leq CA^{n+m}n!^{s_1}m!^{s_2}$, for all $n,m\in\N$. The previous inequalities show that \begin{equation}\label{segment of line}
E[[x_1,x_2]]_s^{(p,q)}=E[[x_1,x_2]]_{(s/p,0)}\cap E[[x_1,x_2]]_{(0,s/q)}\subseteq E[[x_1,x_2]]_{(\lambda s/p,(1-\lambda)s/q)},\hspace{0.2cm} 0\leq \lambda\leq 1.
\end{equation}

\noindent The inclusion follows from the first inequality \begin{equation}\label{prop the segment}\min\{a,b\}\leq a^\lambda b^{1-\lambda}\leq \max\{a,b\},\end{equation} valid for any $a,b>0$ and $0\leq \lambda\leq 1$.

In the analytic setting we use \textit{sectors in the monomial $x_1^px_2^q$}, i.e. sets of the form $$\Pi_{p,q}=\Pi_{p,q}(a,b,r)=\left\{(x_1,x_2)\in \C^2 \hspace{0.1cm}|\hspace{0.1cm} 0<|x_1|^p, |x_2|^q<r,\hspace{0.1cm} a<\text{arg}(x_1^px_2^q)<b\right\}.$$

\noindent Here any convenient branch of arg may be used. The number $r$ denotes the \textit{radius}, $b-a$ the \textit{opening} and $d=(b+a)/2$ the \textit{bisecting direction} of the sector. We will also use the notation $\Pi_{p,q}(a,b,r)=S_{p,q}(d,b-a,r)=S_{p,q}$. In the case $r=+\infty$ we will simply write $\Pi_{p,q}(a,b)=S_{p,q}(d,b-a)$ and we will refer to them as \textit{unbounded sectors}. The definition of subsector in a monomial is clear.

If $f\in\mathcal{O}_b(\Pi_{p,q}(a,b,r),E)$ then we can construct an operator $T_{p,q}(f)_\rho:V(a,b,\rho^2)\rightarrow \mathcal{E}^{(p,q)}_{\rho}$, $0<\rho<r$, related to $\hat{T}_{p,q}$, such that $T_{p,q}(f)_\rho(x_1^px_2^q)(x_1,x_2)=f(x_1,x_2)$. We refer the reader to the original paper \cite{Monomial summ}.

Let $f\in\mathcal{O}(\Pi_{p,q},E)$, $\Pi_{p,q}=\Pi_{p,q}(a,b,r)$ and $\hat{f}\in\mathcal{C}$. We will say that \textit{$f$ has $\hat{f}$ as asymptotic expansion at the origin in $x_1^px_2^q$} (denoted by $f\sim^{(p,q)} \hat{f}$ on $\Pi_{p,q}$) if there is $0<r'\leq r$ such that $\hat{T}_{p,q}\hat{f}=\sum f_nt^n\in \mathcal{E}^{(p,q)}_{r'}[[t]]$ and for every proper subsector  $\Pi_{p,q}'=\Pi_{p,q}(a',b',\rho)$ with $0<\rho<r'$ and $N\in\N$ there exists $C_N(\Pi_{p,q}')>0$ such that \begin{equation}\label{formula def asym x^pe^q}
\left\|f(x_1,x_2)-\sum_{n=0}^{N-1}f_n(x_1,x_2)(x_1^px_2^q)^n\right\|\leq C_N(\Pi_{p,q}')|x_1^px_2^q|^N,\hspace{0.3cm} \text{ on } \Pi_{p,q}'.
\end{equation}

The asymptotic expansion is said to be of \textit{$s-$Gevrey type} (denoted by $f\sim^{(p,q)}_s \hat{f}$ on $\Pi_{p,q}$) if it is possible to choose $C_N(\Pi_{p,q}')=C(\Pi_{p,q}')A(\Pi_{p,q}')^N N!^s$ for some constants $C(\Pi_{p,q}')$, $A(\Pi_{p,q}')$ independent of $N$.

There are different characterizations of asymptotic expansions in a monomial, useful in certain situations. For instance, $f\sim^{(p,q)} \hat{f}$ on $\Pi_{p,q}(a,b,r)$ if and only if $T_{p,q}(f)_\rho\sim \hat{T}_{p,q}(\hat{f})$ on $V(a,b,\rho^2)$, for every $0<\rho<r'$. Here $\hat{T}_{p,q}\hat{f}\in\mathcal{E}_{r'}^{(p,q)}[[t]]$ and $0<r'<r$. The result is also valid for asymptotic expansions of $s-$Gevrey type. We remark that if $f\sim^{(p,q)}_s \hat{f}$ then $\hat{f}\in E[[x_1,x_2]]^{(p,q)}_s$.

In this context the analogous result to Watson's lemma reads as follows: if  $f\sim_s^{(p,q)} 0$ on $\Pi_{p,q}(a,b,r)$ and $b-a>s\pi$ then $f\equiv 0$. Thus we can finally recall the definition of $k-$summability in a monomial.

\begin{defi}Let $\hat{f}\in \mathcal{C}$, let $k>0$ and let $d$ be a direction.
\begin{enumerate}
\item The formal series $\hat{f}$ is called \textit{$k-$summable in $x_1^px_2^q$ on $\Pi_{p,q}=\Pi_{p,q}(a,b,r)$} if $b-a>\pi/k$ and there exists a map $f\in\mathcal{O}(\Pi_{p,q},E)$ such that $f\sim_{1/k}^{(p,q)} \hat{f}$ on $\Pi_{p,q}$. If $d$ is the bisecting direction of $\Pi_{p,q}$ then $\hat{f}$ is called \textit{$k-$summable in $x_1^px_2^q$ in the direction $d$} and $f$ is called \textit{the $k-$sum in $x_1^px_2^q$} of $\hat{f}$ in direction $d$ (or on $\Pi_{p,q}$).

\item The formal series $\hat{f}$ is called \textit{$k-$summable in $x_1^px_2^q$}, if it is $k-$summable in $x_1^px_2^q$ in every direction with finitely many exceptions mod. $2\pi$ (the singular directions).
\end{enumerate}

The space of $k-$summable series in $x_1^px_2^q$ in the direction $d$ will be denoted by $E\{x_1,x_2\}^{(p,q)}_{1/k,d}$ and the space of $k-$summable series in $x_1^px_2^q$ will be denoted by $E\{x_1,x_2\}^{(p,q)}_{1/k}$. Both spaces inherit naturally a structure of algebra when $E$ is a Banach algebra.
\end{defi}

Since $\hat{f}$ is $k-$summable in $x_1^px_2^q$ (resp. $k-$summable in direction $d$) if and only if  $\hat{T}_{p,q}(\hat{f})_\rho$ is $k-$summable (resp. $k-$summable in direction $d$) for all $\rho$ small enough, it is possible to carry on known theorems of summability in one variable in this context.

\section{Monomial Borel and Laplace transforms} \label{BorelyLaplace}

The goal of this section is to introduce an adaptation from \cite{Balser3} of Borel and Laplace transformations in two variables for maps defined on monomial sectors and to develop their main properties. We will also introduce a convolution product naturally associated with these transformations. We start with the study of the Borel transformation.

\begin{defi}The \textit{$k-$Borel transform w.r.t. $x_1^px_2^q$ and weight $(s,1-s)$}, $0<s<1$, of a map $f$ is defined by the formula $$\mathcal{B}_{\boldsymbol{\a}}(f)(\xi_1,\xi_2)=\frac{(\xi_1^{pk}\xi_2^{qk})^{-1}}{2\pi i}\int_\gamma f\left(\xi_1 u^{-\frac{s}{pk}}, \xi_2 u^{-\frac{1-s}{qk}}\right) e^u du,$$

\noindent where $\boldsymbol{\a}=\left(\frac{s}{pk},\frac{1-s}{qk}\right)$ and $\gamma$ denotes a Hankel path as explained below.
\end{defi}

In order to make the formula meaningful, we restrict our attention to maps $f\in\mathcal{O}_b(S_{p,q},E)$ where $S_{p,q}=S_{p,q}(d, \pi/k+2\epsilon, R_0)$, $0<2\epsilon<\pi/k$. Then $\mathcal{B}_{\boldsymbol{\a}}(f)$ is defined and holomorphic on $S_{p,q}(d, 2\epsilon)$. Indeed, if $(\xi_1, \xi_2)\in S_{p,q}(d, 2\epsilon')$, $0<\epsilon'<\epsilon$, take $\gamma$ oriented positively and given by the arc of a circle centered at $0$ and a radius $$R>\max\{(|\xi_1|^p/R_0)^{\frac{k}{s}}, (|\xi_2|^q/R_0)^{\frac{k}{1-s}}\},$$ with endpoints on the directions $-\pi/2-k(\epsilon-\epsilon')$ and $\pi/2+k(\epsilon-\epsilon')$ and the half-lines with those directions from this arc to $\infty$. If $u$ goes along this path the integrand is evaluated on $S_{p,q}$ and the integral converges absolutely, since $f$ is bounded and the exponential term tends to $0$ on those directions. The result is independent of $\epsilon'$ and $R$ due to Cauchy's theorem.

Note that this Borel transform reduces to the usual $k-$Borel transform for maps depending only on the monomial. Using Hankel's formula for the Gamma function we can compute \begin{equation}\label{borel lambda mu}
\mathcal{B}_{\boldsymbol{\a}}(x_1^{\mu_1} x_2^{\mu_2})(\xi_1, \xi_2)=\frac{\xi_1^{\mu_1} \xi_2^{\mu_2}}{\Gamma\left(\frac{\mu_1 s}{pk}+\frac{\mu_2 (1-s) }{qk}\right)}(\xi_1^{pk}\xi_2^{qk})^{-1},\hspace{0.5cm} \mu_1, \mu_2\in\C,
\end{equation}

\noindent and thus introduce $\hat{\mathcal{B}}_{\boldsymbol{\a}}$, the \textit{formal $k-$Borel transform w.r.t. $x_1^px_2^q$ and weight $(s,1-s)$}, defined on $(x_1^px_2^q)^kE[[x_1,x_2]]$. We see from (\ref{segment of line}) that the image of $(x_1^px_2^q)^k E[[x_1,x_2]]^{(p,q)}_{1/k}$ under $\hat{\mathcal{B}}_{\boldsymbol{\a}}$ is included in $E\{\xi_1, \xi_2\}$.

From the integral defining $\mathcal{B}_{\boldsymbol{\a}}$ is natural to consider the vector field $X_{\boldsymbol{\a}}$ given by $$X_{\boldsymbol{\a}}=x_1^{pk}x_2^{qk}\left(\frac{s}{pk}x_1\frac{\d}{\d x_1}+\frac{1-s}{qk}x_2\frac{\d}{\d x_2}\right).$$


If $f\in\mathcal{O}_b(S_{p,q},E)$ as before and $\phi^{\boldsymbol{\a}}_z$ denotes the flow of $X_{\boldsymbol{\a}}$ at time $z$, then it is straightforward to check that
\begin{align}
\label{Borel s1 s2 k y derivadas}
\mathcal{B}_{\boldsymbol{\a}}(X_{\boldsymbol{\a}}(f))(\xi_1, \xi_2)&=(\xi_1^p\xi_2^q)^k\mathcal{B}_{\boldsymbol{\a}}(f)(\xi_1, \xi_2), \\
\label{Borel s1 s2 k y flujo}
\mathcal{B}_{\boldsymbol{\a}}(f\circ\phi_z^{\boldsymbol{\a}})(\xi_1,\xi_2)&=e^{z\xi_1^{pk}\xi_2^{qk}}\mathcal{B}_{\boldsymbol{\a}}(f)(\xi_1,\xi_2).
\end{align}

Both formulas are naturally related and we can interpret (\ref{Borel s1 s2 k y derivadas}) as the linearized of (\ref{Borel s1 s2 k y flujo}) at $z=0$. In the variable $t=x_1^px_2^q$ the vector field $X_{\boldsymbol{\a}}$ reduces to $t^{k+1}/k \d/\d t$, and formula (\ref{Borel s1 s2 k y derivadas}) is just $\mathcal{B}_k\left(t^{k+1}\frac{\d F}{\d t}\right)(\zeta)=k\zeta^k\mathcal{B}_k(F)(\zeta)$.

It is convenient to introduce the following terminology: we will say that $f$ is \textit{of exponential growth at most $(\a_1, \a_2)\in\R_{>0}^2$ on $\Pi_{p,q}(a,b)$} if for every unbounded subsector $\Pi_{p,q}'$ of $\Pi_{p,a}(a,b)$ there are positive constants $C,M$ such that \begin{equation}\label{exp growth k, s_1 s_2}
\|f(\xi_1,\xi_2)\|\leq Ce^{M\left(|\xi_1|^{\a_1}+|\xi_2|^{\a_2}\right)},\hspace{0.4cm} \text{on }\Pi_{p,q}'.
\end{equation}

The term $|\xi_1|^{\a_1}+|\xi_2|^{\a_2}$ can be interchanged by $\max\{|\xi_1|^{\a_1},|\xi_2|^{\a_2}\}$ and vice versa by changing $M$ above. We note that if $f$ is entire and $\sum a_{n,m}x_1^n x_2^m$ is its Taylor series at the origin then the above condition is equivalent to have bounds of the type $$\|a_{n,m}\|\leq \frac{C' A^{n+m}}{\Gamma\left(1+\frac{n}{\a_1}+\frac{m}{\a_2}\right)}, \text{ for all } n,m\in\N.$$ This can be deduced using Cauchy's integral formulas for the coefficients, Stirling's formula and the inequality \begin{equation}\label{inq Gamma}
\Gamma(1+a)\Gamma(1+b)\leq\Gamma(1+a+b),\hspace{0.3cm}a,b>0,
\end{equation} satisfied by the Gamma function.

\begin{nota}\label{Borel an T_pq}If $\hat{f}\in \mathcal{C}$, $\hat{\varphi}_{\boldsymbol{\a}}=\hat{\mathcal{B}}_{\boldsymbol{\a}}(x_1^{pk}x_2^{qk}\hat{f})$,  $\hat{T}_{p,q}(\hat{f})=\sum f_n t^n$ and $\hat{T}_{p,q}(\hat{\varphi}_{\boldsymbol{\a}})=\sum \varphi_n \tau^n$ then $f_n$ and $\varphi_n$ are related by $\xi_1^{pn}\xi_2^{qn}\varphi_n(\xi_1,\xi_2)=\mathcal{B}_{\boldsymbol{\a}}(x_1^{p(n+k)}x_2^{q(n+k)}f_n)$. Since the $f_n$ are holomorphic in a common polydisc $D_\rho^2$ we see that the $\varphi_n$ are all entire maps and there are positive constants $L, M$  independent of $n$ but depending on $\rho$ such that \begin{equation}\label{growth of phi_m}
\|\varphi_n(\xi_1,\xi_2)\|\leq \frac{L\|f_n\|_\rho}{\Gamma\left(1+\frac{n}{k}\right)}e^{M\left(|\xi_1|^{\frac{pk}{s}}+|\xi_2|^{\frac{qk}{1-s}}\right)}, \text{ on } \C^2.
\end{equation}
\end{nota}

For further references we prove in the following proposition the behavior of the Borel transform w.r.t. monomial asymptotic expansions when the corresponding monomials are the same: the statement resembles the corresponding one in \cite{Balser3} and the proof mimics the proof of the behavior of a Borel transform for several variables included in \cite{Sanz}.

\begin{prop}\label{exp growth of Borel k, s1 s2}Consider $f\in\mathcal{O}(S_{p,q}(d,\pi/k+2\epsilon, R_0),E)$, $0<2\epsilon<\pi/k$, $\hat{f}\in\mathcal{C}$ and $s_1\geq0$. Let $g=\mathcal{B}_{\boldsymbol{\a}}(x_1^{pk}x_2^{qk} f)$, $\hat{g}=\hat{\mathcal{B}}_{\boldsymbol{\a}}(x_1^{pk}x_2^{qk}\hat{f})$ and $\hat{T}_{p,q}(\hat{g})=\sum g_n t^n$, where $\boldsymbol{\a}=\left(\frac{s}{pk},\frac{1-s}{qk}\right)$ and $0<s<1$. If $f\sim_{s_1}^{(p,q)}\hat{f}$ on $S_{p,q}$ then $g\sim^{(p,q)}_{s_2} \hat{g}$ on $S_{p,q}(d,2\epsilon)$, where $s_2=\max\left\{s_1-\frac{1}{k},0\right\}$. Furthermore for every unbounded subsector $S_{p,q}''$ of $S_{p,q}(d,2\epsilon)$ there are positive constants $B,D,M$ such that $$\left\|g(\xi_1,\xi_2)-\sum_{n=0}^{N-1} g_n(\xi_1,\xi_2) (\xi_1^p\xi_2^q)^n\right\|\leq DB^N\Gamma(1+Ns_2)|\xi_1^p\xi_2^q|^Ne^{M\left(|\xi_1|^{\frac{pk}{s}}+|\xi_2|^{\frac{qk}{1-s}}\right)}\text{ on }S_{p,q}''.$$ For $N=0$ the inequality is interpreted as $g$ being of exponential growth at most  $\left(\frac{pk}{s},\frac{qk}{1-s}\right)$ on $S_{p,q}(d,2\epsilon)$.
\end{prop}

\begin{proof}We have to establish the bounds for sectors of the form $S_{p,q}'=S_{p,q}(d,2\epsilon')$ with $0<\epsilon'<\epsilon$. The proof relies on choosing adequately the radius of the arc of the path $\gamma$ in the definition. Write $\gamma=\gamma_1+\gamma_2-\gamma_3$ where $\gamma_1, \gamma_3$ denote the half-lines and $\gamma_2$ denotes the circular part, parameterized by $\gamma_2(\theta)=Re^{i\theta}$, $|\theta|\leq\pi/2+k(\epsilon''-\epsilon')/2$, where $0<\epsilon'<\epsilon''<\epsilon$ and $R$ will be chosen so that $(\xi_1 u^{-\frac{s}{pk}}, \xi_2 u^{-\frac{1-s}{qk}})\in S_{p,q}=S_{p,q}(d,\pi/k+2\epsilon'', R_0/2)$ for all $u$ on $\gamma$ and $(\xi_1,\xi_2)\in S_{p,q}'$.

Without loss of generality we may assume that the $f_n$ are holomorphic on $D_{R_0}^2$. By hypothesis inequality (\ref{formula def asym x^pe^q}) holds with $C_N=CA^N\Gamma(1+Ns_1)$ on $S_{p,q}$. Let us set $a=\sin\left(\frac{k}{2}(\epsilon''-\epsilon')\right)$ and with $R$ to be chosen, a straightforward estimate using Remark \ref{Borel an T_pq} shows that for all $N\in\N$ we have \begin{equation}\label{proof Bks_1s_2 2}
\left\|g(\xi_1,\xi_2)-\sum_{n=0}^{N-1}g_n(\xi_1,\xi_2)(\xi_1^p\xi_2^q)^n\right\|\leq
\frac{2CA^N}{a}\Gamma(1+Ns_1)\frac{|\xi_1^p\xi_2^q|^N}{R^{N/k}}\left(\frac{e^{-aR}}{R}+e^R\right)\text{ on }S_{p,q}'.
\end{equation}

\noindent For $N=0$ we are denoting $C=\sup_{(x_1,x_2)\in S_{p,q}}\|f(x_1,x_2)\|$ (note that $f$ is bounded here, as we have reduced radius and opening of the sector).

To prove the statement of the Proposition, let us divide our sector in two parts. First of all,  consider $(\xi_1,\xi_2)\in S_{p,q}(d,2\epsilon',r)$, with $r>0$ fixed. In this case choose $R\geq \max\{(2r/R_0)^{\frac{k}{s}}, (2r/R_0)^{\frac{k}{1-s}}\}$. Since it is enough to establish the bounds for large $N$ we can suppose $N$ is large enough and take $R=N/k$. Then the bound follows using Stirling's formula twice, as we have, asymptotically, that
$$
\frac{e^{R}}{R^{N/k}} = \frac{e^{N/k}}{\left( N/k \right)^{N/k}} \sim \frac{\sqrt{2\pi N/k}}{\Gamma (1+N/k)}.
$$
and that $\dfrac{\Gamma (1+Ns_1)}{\Gamma (1+N/k)}$ can be bounded by $C' (A')^{N} \Gamma ( 1+N (s_1-\frac{1}{k}))$, if $s_1\geq 1/k$, for some $C'$, $A'>0$.

To finish we establish the bound in the complementary, i.e., for $(\xi_1, \xi_2)\in S_{p,q}'\setminus S_{p,q}(d,2\epsilon',r)$. If $R_1, R_2<R_0/2$ we can take $R=R(\xi_1,\xi_2)=\max\{(|\xi_1|^p/R_1)^{\frac{k}{s}}, (|\xi_2|^q/R_2)^{\frac{k}{1-s}}\}$. Using the second inequality of (\ref{prop the segment}) we see that
$\frac{|\xi_1^p\xi_2^q|^k}{R_1^k R_2^k}\leq R(\xi_1,\xi_2)$ and thus (\ref{proof Bks_1s_2 2}) is bounded by \begin{equation}\label{last inequatily proof Borel}
\frac{CA^N}{a}\Gamma(1+Ns_1)R_1^N R_2^N\left(\frac{e^{-a\left(\frac{|\xi_1|^p}{R_1}\right)^{\frac{k}{s}}}}{\left(\frac{|\xi_1|^p}{R_1}\right)^k}+e^{\left(\frac{|\xi_1|^p}{R_1}\right)^{\frac{k}{s}}}\right)\left(\frac{e^{-a\left(\frac{|\xi_2|^q}{R_2}\right)^{\frac{k}{1-s}}}}{\left(\frac{|\xi_2|^q}{R_2}\right)^k}+e^{\left(\frac{|\xi_2|^q}{R_2}\right)^{\frac{k}{1-s}}}\right).
\end{equation}

Arguing as in \cite[Thm. 4.8]{Sanz} we see that for $k, a, S>0$ and $\rho>r>0$ there are positive constants $L, K, M'$ such that \begin{equation}\label{h(N,t)}
\inf_{0<t<\rho} t^N\left(\frac{e^{-a(\tau/t)^{\frac{k}{S}}}}{(\tau/t)^k}+e^{(\tau/t)^{\frac{k}{S}}}\right)\leq \frac{LK^N}{\Gamma\left(1+\frac{NS}{k}\right)}\tau^N e^{M'\tau^{\frac{k}{S}}}, \hspace{0.2cm} \text{ for all }\tau>r, N\in \N.
\end{equation}

Since $R_1$ and $R_2$ can be arbitrarily small, inequality (\ref{h(N,t)}) can be applied to $\tau=|\xi_1|^p$, $S=s$ and to $\tau=|\xi_2|^q$, $S=1-s$ to conclude that there are large enough constants $D',B',M$ such that $$\left\|g(\xi_1,\xi_2)-\sum_{n=0}^{N-1}g_n(\xi_1,\xi_2)(\xi_1^p\xi_2^q)^n\right\|\leq D'B'^N \frac{\Gamma(1+Ns_1)}{\Gamma\left(1+\frac{sN}{k}\right)\Gamma\left(1+\frac{(1-s)N}{k}\right)} |\xi_1^p\xi_2^q|^Ne^{M\left(|\xi_1|^{\frac{pk}{s}}+|\xi_2|^{\frac{qk}{1-s}}\right)},$$

\noindent on $S_{p,q}'\setminus S_{p,q}(d,2\epsilon',r)$. We can use Stirling's formula to finally conclude the proof.
\end{proof}

We now introduce the corresponding Laplace transformation that will turn out to be the inverse of the Borel transformation.

\begin{defi}The \textit{$k-$Laplace transform w.r.t. $x_1^px_2^q$ with weight $(s,1-s)$ in direction $\theta$}, $0<s<1$, $|\theta|<\pi/2$, of a map $f$ is defined by the formula $$\mathcal{L}_{\boldsymbol{\a},\theta}(f)(x_1, x_2)=x_1^{pk}x_2^{qk}\int_{0}^{e^{i\theta}\infty} f\left(x_1 u^{\frac{s}{pk}}, x_2 u^{\frac{1-s}{qk}}\right) e^{-u} du.$$ Here as before $\boldsymbol{\a}=\left(\frac{s}{pk},\frac{1-s}{qk}\right)$.
\end{defi}

We restrict our attention to maps $f\in\mathcal{O}(\Pi_{p,q},E)$,  $\Pi_{p,q}=\Pi_{p,q}(a,b)$. One may be tempted to impose an exponential growth on $f$ of order $k$ in the monomial $\xi_1^p\xi_2^q$, i.e., to suppose that $\|f(\xi_1,\xi_2)\|\leq Ce^{M|\xi_1^p\xi_2^q|^k}$ on $\Pi_{p,q}$ but this would imply that $f$ is a map depending only on $\xi_1^p\xi_2^q$. Indeed, if there exist a positive real function $K$ such that $\| f(\xi_1,\xi_2)\|\leq K(|\xi_1^p\xi_2^q |)$ on $\Pi_{p,q}$, $f$ can be decomposed as
 $f(\xi_1,\xi_2)=\sum_{{0\leq i<p}\atop {0\leq j <q}} \xi_1^i\xi_2^{j} f_{ij} (\xi_1^p,\xi_2^q)$. Consider now
 $$
 T_{1,1}(f_{ij}(t))(\zeta_1,\zeta_2)= \sum_{m=0}^{\infty} f_{ij,m} (t)\zeta_2^m +\sum_{m=1}^{\infty} \frac{f_{ij,-m}(t)}{t^m} \zeta_1^m,
 $$
 as in \cite{CM}. Using Cauchy formulas, if $|t|<\rho^2$ we have that
 \begin{align}
 \|f_{ij,m}(t)\| & \leq \frac{K(|t|)}{|t|^{i/p} \rho^{m+j/q-i/p}},\quad \text{if } m\geq 0,\\
 \|f_{ij,-m}(t)\| & \leq \frac{|t|^m K(|t|)}{|t|^{j/q} \rho^{m+i/p-j/q}},\quad \text{if } m>0.
 \end{align}
 If $m\geq 1$, $m+j/q-i/p>0$, $m+i/p-j/q>0$, and so, letting $\rho \rightarrow +\infty$, we obtain that $f_{ij,m}(t)\equiv 0$. If $m>0$ and $(i,j)\neq (0,0)$, consider one of the preceding inequalities, according to $j/q>i/p$ or $j/q<i/p$, and the same conclusion follows. Finally, we obtain that $f(\xi_1,\xi_2)=f_{00}(\xi_1^p\xi_2^q)$.

So, instead of imposing such an exponential growth to $f$, and as it is suggested by Proposition \ref{exp growth of Borel k, s1 s2}, it is natural to assume that $f$ is of exponential growth of order at most $\left(\frac{pk}{s}, \frac{qk}{1-s}\right)$ on $\Pi_{p,q}$. In such case, $\mathcal{L}_{\boldsymbol{\a},\theta}(f)$ is defined and holomorphic on the domain $D_{k,\theta}(a,b,M)$ of $\C^2$ defined by the conditions $$a-\theta/k<\text{arg}(x_1^px_2^q)<b-\theta/k,\hspace{0.2cm}M\left(|x_1|^{\frac{pk}{s}}+|x_2|^{\frac{qk}{1-s}}\right)<\cos\theta.$$

We note that by changing the direction $\theta$ by $\theta'$ we obtain an analytic continuation of $\mathcal{L}_{\boldsymbol{\a},\theta}(f)$ when $|\theta'-\theta|<k(b-a)$, a fact that follows directly from Cauchy's theorem. This process leads to an holomorphic map $\mathcal{L}_{\boldsymbol{\a}}(f)$ defined in the region $D_k(a,b,M)=\bigcup_{|\theta|<\pi/2} D_{k,\theta}(a,b,M)$. Note that given $a-\pi/2k<a'<b'<b+\pi/2k$ there is $r>0$ such that $\Pi_{p,q}(a',b',r)\subset D_k(a,b,M)$. As before, this Laplace transform reduces to the usual $k-$Laplace transform for maps depending only on the monomial.

The usual method of proof, i.e. an adequate choice of the path $\gamma$ in the definition of $\mathcal{B}_{\boldsymbol{\a}}$ and the Residue theorem, let us show that if $f\in\mathcal{O}_b(S_{p,q}(d,\pi/k+2\epsilon,R_0),E)$, $0<2\epsilon<\pi/k$ then $$\mathcal{L}_{\boldsymbol{\a}}\mathcal{B}_{\boldsymbol{\a}}(x_1^{pk}x_2^{qk}f)=x_1^{pk}x_2^{qk}f, \text{ on the intersection of their domains.}$$

An straightforward adaptation of Lerch's theorem shows that $\mathcal{L}_{\boldsymbol{\a}}$ is in fact injective. It follows that if $g$ is of exponential growth of order at most $\left(\frac{pk}{s}, \frac{qk}{1-s}\right)$ then $$\mathcal{B}_{\boldsymbol{\a}}\mathcal{L}_{\boldsymbol{\a}}(g)=g, \text{ on the intersection of their domains.}$$

We may define $\hat{\mathcal{L}}_{\boldsymbol{\a}}$, \textit{the formal $k-$Laplace transform w.r.t. $x_1^px_2^q$ and weight $(s,1-s)$}, as the inverse of $\hat{\mathcal{B}}_{\boldsymbol{\a}}$. Before we state the behavior of $\mathcal{L}_{\boldsymbol{\a}}$ w.r.t. monomial asymptotic expansion in the same monomial, we describe in the following remark the information we need about series such that their formal and analytic Laplace transforms coincide.

\begin{nota}\label{nota sobre Laplace formal y analitico}Let $\hat{f}=\sum a_{n,m}\xi_1^n\xi_2^m$ be a formal power series and $\hat{T}_{p,q}(\hat{f})=\sum f_n\tau^n$. A necessary and sufficient condition on $\hat{f}$  so that $\hat{\mathcal{L}}_{\boldsymbol{\a}}(\hat{f})$ is convergent is that $\hat{f}$ defines an entire map $f$ of exponential growth of order at most $\left(\frac{pk}{s}, \frac{qk}{1-s}\right)$ on $\C^2$. Then $\frac{1}{x_1^{pk}x_2^{qk}}\mathcal{L}_{\boldsymbol{\a}}(f)$ is holomorphic in a polydisc at the origin and $\frac{1}{x_1^{pk}x_2^{qk}}\hat{\mathcal{L}}_{\boldsymbol{\a}}(\hat{f})$ is its Taylor's series.

Now assume that there are positive constants $s_1,B,D,M$ such that the family of maps $f_n$ are entire and satisfy the bounds $$\|f_n(\xi_1,\xi_2)\|\leq DB^n\Gamma\left(1+ns_1\right)e^{M\left(|\xi_1|^{\frac{pk}{s}}+|\xi_2|^{\frac{qk}{1-s}}\right)},\hspace{0.3cm}\text{ on }\C^2.$$

Then we can conclude that $\hat{f}\in E[[\xi_1,\xi_2]]_{s_1}^{(p,q)}$, $\hat{\mathcal{L}}_{\boldsymbol{\a}}(\hat{f})\in x_1^{pk}x_2^{qk}E[[x_1,x_2]]_{s_1+1/k}^{(p,q)}$, all the maps $\mathcal{L}_{\boldsymbol{\a}}(f_n)$ are holomorphic in a common polydisc centered at the origin and $$\hat{\mathcal{L}}_{\boldsymbol{\a}}(\hat{f})=\sum_{n\geq 0} \mathcal{L}_{\boldsymbol{\a}}((\xi_1^p\xi_2^q)^nf_n ).$$

\end{nota}

We focus now in the behavior of the Laplace transform w.r.t. monomial asymptotic expansions. The reader may note that the hypotheses required may seem restrictive, but in fact those appear naturally when we compare with Proposition \ref{exp growth of Borel k, s1 s2}.

\begin{prop}\label{Laplace k s_1 s_2 properties} Consider $f\in \mathcal{O}(\Pi_{p,q}(a,b),E)$, $\Pi_{p,q}=\Pi_{p,q}(a,b)$, $\hat{f}\in \mathcal{C}$ and $s_1\geq0$. Assume that $f\sim^{(p,q)}_{s_1} \hat{f}$ on $\Pi_{p,q}$ and additionally:\begin{enumerate}
\item If $\hat{T}_{p,q}(\hat{f})=\sum f_n t^n$ then every $f_n$ is an entire map and there are positive constants $B,D,K$ such that $$\|f_n(\xi_1,\xi_2)\|\leq DB^n\Gamma\left(1+ns_1\right)e^{K\left(|\xi_1|^{\frac{pk}{s}}+|\xi_2|^{\frac{qk}{1-s}}\right)},\text{ on }\C^2.$$
\item For every unbounded subsector $\Pi_{p,q}'$ of $\Pi_{p,q}$ there are positive constants $C,A,M$ such that for all $N\in\N$ \begin{equation*}
\left\|f(\xi_1,\xi_2)-\sum_{n=0}^{N-1} f_n(\xi_1,\xi_2) (\xi_1^p\xi_2^q)^n\right\|\leq CA^N\Gamma(1+Ns_1)|\xi_1^p\xi_2^q|^Ne^{M\left(|\xi_1|^{\frac{pk}{s}}+|\xi_2|^{\frac{qk}{1-s}}\right)}\text{ on }\Pi_{p,q}'.\end{equation*}
\end{enumerate}
\noindent Then $\frac{1}{x_1^{pk}x_2^{qk}}\mathcal{L}_{\boldsymbol{\a}}(f)\sim_{s_2}^{(p,q)} \frac{1}{x_1^{pk}x_2^{qk}}\hat{\mathcal{L}}_{\boldsymbol{\a}}(\hat{f})$ on $D_{k}(a,b,M)$, where $s_2=\frac{1}{k}+s_1$.
\end{prop}

\begin{proof}We note that hypothesis (2) for $N=0$ is interpreted as $f$ being of exponential growth of order at most $\left(\frac{pk}{s}, \frac{qk}{1-s}\right)$ on $\Pi_{p,q}$. Let us write $R=M\left(|x_1|^{\frac{pk}{s}}+|x_2|^{\frac{qk}{1-s}}\right)$, $x_1^{pk}x_2^{qk}h=\mathcal{L}_{\boldsymbol{\a}}(f)$ and $\hat{T}_{p,q}\left(\frac{1}{x_1^{pk}x_2^{qk}}\hat{\mathcal{L}}_{\boldsymbol{\a}}(\hat{f})\right)=\sum h_n t^n$. Then using (1) and Remark \ref{nota sobre Laplace formal y analitico} we see that $(x_1^px_2^q)^{n+k}h_n=\mathcal{L}_{\boldsymbol{\a}}((\xi_1^p\xi_2^q)^nf_n)$.

Fixing $|\theta|<\pi/2$ it is enough to prove the result for subsectors contained in $D_{k,\theta}(a,b,M)$. If $\Pi_{p,q}''$ is one of them we can find $\delta>0$ small enough such that $R<\cos\theta-\delta,$ on $\Pi_{p,q}''$. Now let $\Pi_{p,q}'$ be a subsector of $\Pi_{p,q}$ such that $(x_1u^{\frac{s}{pk}}, x_2u^{\frac{1-s}{qk}})\in \Pi_{p,q}'$ if $(x_1,x_2)\in \Pi_{p,q}''$ and $u$ is on the half-line from $0$ to $\infty$ in direction $\theta$ . Using statement $(2)$ for $\Pi_{p,q}'$ we see that $$\left\|h(x_1,x_2)-\sum_{n=0}^{N-1}h_n(x_1,x_2)(x_1^px_2^q)^n\right\|\leq\frac{C}{\delta}\frac{A^N}{\delta^{N/k}}\Gamma(1+Ns_1)\Gamma\left(1+\frac{N}{k}\right)|x_1^px_2^q|^N,\text{ on }\Pi_{p,q}''.$$

An application of inequality (\ref{inq Gamma}) leads us to the result.
\end{proof}

Finally we introduce a natural convolution product that in our context shares similar properties with the classical one. Indeed, the \textit{convolution between $f$ and $g$ w.r.t. $x_1^px_2^q$ and weight $(s,1-s)$}, where $\boldsymbol{\a}=\left(\frac{s}{pk},\frac{1-s}{qk}\right)$ and $0<s<1$ is defined by $$(f\ast_{\boldsymbol{\a}} g)(x_1,x_2)=x_1^{pk}x_2^{qk}\int_0^1 f(x_1\tau^{\frac{s}{pk}},x_2\tau^{\frac{1-s}{qk}})g(x_1(1-\tau)^{\frac{s}{pk}},x_2(1-\tau)^{\frac{1-s}{qk}})d\tau.$$

We remark that this formula is already included implicitly in \cite{BalserMozo}, \cite{Balser3}. Some calculations show that this operation is bilinear, commutative and associative. With the aid of the Beta function we obtain the following formula valid for $\lambda_1, \lambda_2, \mu_1, \mu_2\in\C$ with positive real part $$\frac{x_1^{\lambda_1}x_2^{\mu_1}}{\Gamma\left(\frac{s}{pk}\lambda_1+\frac{1-s}{qk}\mu_1+1\right)}\ast_{\boldsymbol{\a}} \frac{x_1^{\lambda_2}x_2^{\mu_2}}{\Gamma\left(\frac{s}{pk}\lambda_2+\frac{1-s}{qk}\mu_2+1\right)}=\frac{x_1^{\lambda_1+\lambda_2+pk}x_2^{\mu_1+\mu_2+qk}}{\Gamma\left(\frac{s}{pk}(\lambda_1+\lambda_2)+\frac{1-s}{qk}(\mu_1+\mu_2)+2\right)}.$$

If $f,g$ are of exponential growth of order at most $\left(\frac{pk}{s}, \frac{qk}{1-s}\right)$ (resp. belong to $\mathcal{O}_b(S_{p,q},E)$) then the same is valid for $f\ast_{\boldsymbol{\a}}g$ and \begin{equation}\label{conv. product}
\mathcal{L}_{\boldsymbol{\a}}(f\ast_{\boldsymbol{\a}} g)=\mathcal{L}_{\boldsymbol{\a}}(f)\mathcal{L}_{\boldsymbol{\a}}(g), \hspace{0.5cm} \text{(resp. }\mathcal{B}_{\boldsymbol{\a}}(fg)=\mathcal{B}_{\boldsymbol{\a}}(f)\ast_{\boldsymbol{\a}}\mathcal{B}_{\boldsymbol{\a}}(g)\text{)},
\end{equation} as in the classical case.

\section{Monomial Borel-Laplace summation methods} \label{sumacion}

We can define a summation method based in the above Borel and Laplace transforms and we will see that it turns out to be equivalent to monomial summability.

\begin{defi}\label{def k-s_1,s_2-Borel sum}Assume $\hat{f}\in E[[x_1,x_2]]_{1/k}^{(p,q)}$ and set  $\hat{\varphi}_{s}=\hat{\mathcal{B}}_{\boldsymbol{\a}}(x_1^{pk}x_2^{qk}\hat{f})$, where $\boldsymbol{\a}=\left(\frac{s}{pk},\frac{1-s}{qk}\right)$ and $0<s<1$.  We will say that $\hat{f}$ is $k-(s,1-s)-$\textit{Borel summable in the monomial $x_1^px_2^q$ in direction} $d$ if $\hat{\varphi}_{s}$ can be analytically continued, say as $\varphi_{s}$, to a monomial sector of the form $S_{p,q}(d,2\epsilon)$ and being of exponential growth of order at most $\left(\frac{pk}{s}, \frac{qk}{1-s}\right)$ there. In this case the $k-(s,1-s)-$\textit{Borel sum of} $\hat{f}$ in direction $d$ is defined as $$f(x_1,x_2)=\frac{1}{x_1^{pk}x_2^{qk}}\mathcal{L}_{\boldsymbol{\a}}(\varphi_{s})(x_1,x_2).$$
\end{defi}

It is worth to note that the factor $x_1^{pk}x_2^{qk}$ that multiply the formal power series in the previous definition is used to avoid the use of power series with non-integer exponents as we have done so far. We remark the above notion is not modified if we remove this factor and the corresponding fraction in the $k-(s,1-s)-$Borel sum. Then we can prove directly from the definition that this summation method is stable by sums, products in case $E$ being a Banach algebra (using the convolution $\ast_{\boldsymbol{\a}}$, $\boldsymbol{\a}=\left(\frac{s}{pk},\frac{1-s}{qk}\right)$) and by the derivation $\frac{s}{pk} x_1\frac{\d}{\d x_1}+\frac{1-s}{qk} x_2\frac{\d}{\d x_2}$ (using formula (\ref{Borel s1 s2 k y derivadas})). Instead we compare this method with $k-$summability in the monomial $x_1^px_2^q$ in direction $d$.

Let $\hat{f}\in E\{x_1,x_2\}_{1/k,d}^{(p,q)}$ and $f\in\mathcal{O}(S_{p,q}(d,\pi/k+2\epsilon, R_0),E)$ its $k-$sum in $x_1^px_2^q$ in direction $d$. For a fixed $0<s<1$ put  $\varphi_{s}=\mathcal{B}_{\boldsymbol{\a}}(x_1^{pk}x_2^{qk} f)$ and $\hat{\varphi}_{s}=\hat{\mathcal{B}}_{\boldsymbol{\a}}(x_1^{pk}x_2^{qk}\hat{f})$.
We know that $\hat{\varphi}_{s}$ is a convergent power series in some $D_r^2$ and by Proposition \ref{exp growth of Borel k, s1 s2} we see that $\varphi_{s}\sim_0^{(p,q)} \hat{\varphi}_{s}$ on $S_{p,q}(d,2\epsilon)$. These two properties imply that $\varphi_{s}$ coincides with the sum of $\hat{\varphi}_{s}$ on $S_{p,q}(d,2\epsilon)\cap D_r^2$. In other words $\hat{\varphi}_{s}$ can be analytically continued with exponential growth of order at most $\left(\frac{pk}{s}, \frac{qk}{1-s}\right)$ on $S_{p,q}(d,2\epsilon)$ and therefore $\hat{f}$ is $k-(s,1-s)-$Borel summable in the monomial $x_1^px_2^q$ in direction $d$. Since $\mathcal{B}_{\boldsymbol{\a}}$ and $\mathcal{L}_{\boldsymbol{\a}}$ are inverses of each other the $k-(s,1-s)-$Borel sum of $\hat{f}$ is $f$.

Conversely, let $\hat{f}$ be $k-(s,1-s)-$Borel summable in the monomial $x_1^px_2^q$ in direction $d$ and let $\varphi_{s}\in\mathcal{O}(S_{p,q}(d, 2\epsilon),E)$ and $\hat{\varphi}_{s}$ be as in Definition \ref{def k-s_1,s_2-Borel sum}. Let us write $R(\xi_1,\xi_2)=|\xi_1|^{\frac{pk}{s}}+|\xi_2|^{\frac{qk}{1-s}}$ and $\hat{T}_{p,q}(\hat{\varphi}_s)=\sum \varphi_n \tau^n$. By hypothesis there are positive constants $D',M'$ such that $\|\varphi_s(\xi_1,\xi_2)\|\leq D'e^{M'R(\xi_1,\xi_2)},$ on $S_{p,q}(d,2\epsilon')$, $0<\epsilon'<\epsilon$ and from Remark \ref{Borel an T_pq} we know that all the $\varphi_n$ are entire maps and there are positive constants $D,B,M$ such that $\|\varphi_n(\xi_1,\xi_2)\|\leq DB^n e^{MR(\xi_1,\xi_2)},$ on $\C^2$. Enlarging the constants if necessary we may assume that $D=D'$ and $M=M'$.

To be able to apply Proposition \ref{Laplace k s_1 s_2 properties}, it is enough to prove that there are positive constants $C,A$ such that for all $N\in\N$ we have \begin{equation}\label{cond eq BL y MS}
\left\|\varphi_{s}(\xi_1,\xi_2)-\sum_{n=0}^{N-1} \varphi_n(\xi_1,\xi_2) (\xi_1^p\xi_2^q)^n\right\|\leq CA^N|\xi_1^p\xi_2^q|^Ne^{MR(\xi_1,\xi_2)},\hspace{0.3cm} \text{ on }S_{p,q}(d,2\epsilon').\end{equation}

Since $\hat{\varphi}_{s}$ is the convergent Taylor's series of $\varphi_{s}$ at the origin then (\ref{cond eq BL y MS}) is satisfied for all $|\xi_1|, |\xi_2|\leq R$ for some $R>0$. Additionally, due to the growth of the functions $\varphi_n$, the series of maps $\sum \varphi_n(\xi_1, \xi_2) (\xi_1^p\xi_2^q)^n$ converges absolutely in compacts of the set of $(\xi_1,\xi_2)$ such that $B|\xi_1^p\xi_2^q|<1$. Then $\varphi_{s}$ can be analytically continued there through this series. It follows that if $B|\xi_1^p\xi_2^q|<1/2$ then inequality (\ref{cond eq BL y MS}) is also satisfied. On the other hand the inequalities of the previous paragraph show that the left side of (\ref{cond eq BL y MS}) is bounded by  $$De^{MR(\xi_1,\xi_2)}+\sum_{n=0}^{N-1} DB^n |\xi_1^p\xi_2^q|^ne^{MR(\xi_1,\xi_2)}\hspace{0.3cm} \text{ on } S_{p,d}(d,2\epsilon').$$

\noindent If $1/2\leq B|\xi_1^p\xi_2^q|\leq 2$ the last expression is bounded by $2^NDe^{MR(\xi_1,\xi_2)}\leq D(4B)^N|\xi_1^p\xi_2^q|^Ne^{MR(\xi_1,\xi_2)}$. If $B|\xi_1^p\xi_2^q|>2$, we can bound it by $$De^{MR(\xi_1,\xi_2)}+D\frac{B^N|\xi_1^p\xi_2^q|^N-1}{B|\xi_1^p\xi_2^q|-1}e^{MR(\xi_1,\xi_2)}< D B^N|\xi_1^p\xi_2^q|^Ne^{MR(\xi_1,\xi_2)},$$ and thus (\ref{cond eq BL y MS}) is valid in all the cases with $C,A$ large enough. Applying Proposition \ref{Laplace k s_1 s_2 properties} to $\varphi_{s}$, $\hat{\varphi}_s$ and $s_1=0$ we conclude that $$f(x_1,x_2)=\frac{1}{x_1^{pk}x_2^{qk}}\mathcal{L}_{\boldsymbol{\a}}(\varphi_{s})\sim_{1/k}^{(p,q)} \frac{1}{x_1^{pk}x_2^{qk}}\hat{\mathcal{L}}_{\boldsymbol{\a}}(\hat{\varphi}_s)=\hat{f}, \hspace{0.3cm} \text{ on } D_k(d-\epsilon, d+\epsilon,M).$$

In conclusion, $\hat{f}$ is $k-$summable in $x_1^px_2^q$ in direction $d$ and the $k-$sum can be found through the $k-$Laplace transform w.r.t. $x_1^px_2^q$ and weight $(s,1-s)$ of $\varphi_s$. These considerations prove the following theorem, we consider one of the main results of this work.

\begin{teor}\label{Monomial summability t Borel Laplace}Let $\hat{f}$ be a $1/k-$Gevrey series in the monomial $x_1^px_2^q$. Then it is equivalent:
\begin{enumerate}
\item $\hat{f}\in E\{x_1,x_2\}_{1/k,d}^{(p,q)}$,

\item There is $0<s<1$ such that $\hat{f}$ is $k-(s,1-s)-$Borel summable in the monomial $x_1^px_2^q$ in direction $d$.

\item For all $0<s<1$, $\hat{f}$ is $k-(s,1-s)-$Borel summable in the monomial $x_1^px_2^q$ in direction $d$.
\end{enumerate}
In all cases the corresponding sums coincide.
\end{teor}

We remark that some properties of monomial summability can be obtained easily from this theorem. For instance $E\{x_1.x_2\}^{(p,q)}_{1/k,d}=E\{x_1.x_2\}^{(Np,Nq)}_{N/k,Nd}$ for all $p,q,N\in \N^*$ and all directions $d$. Indeed, for any $0<s<1$, $\mathcal{B}_{\boldsymbol{\a}}$ and $\mathcal{L}_{\boldsymbol{\a}}$ are independent of $N$, since $\boldsymbol{\a}=\left(\frac{s}{pk},\frac{1-s}{qk}\right)=\left(\frac{s}{(Np)\frac{k}{N}},\frac{1-s}{(Nq)\frac{k}{N}}\right)$. It can also be shown directly that if $\hat{f}\in E\{x_1,x_2\}^{(p,q)}_{1/k}$ has no singular directions then it is convergent.



\section{Application to a family of partial differential equations} \label{ejemplos}

As mentioned in the introduction monomial summability is a useful tool to understand summability properties of solutions of doubly singular equations, i.e. differential equations of type \begin{equation}\label{SPDE}
\varepsilon^q x^{p+1}\frac{\d \textbf{y}}{\d x}=\textbf{F}(x,\varepsilon,\textbf{y}),\hspace{0.3cm} \textbf{F}(0,0,\textbf{0})=\textbf{0},
\end{equation} where $p,q\in\N_{>0}$ and $\textbf{F}$ is a $\C^l-$valued holomorphic map in some neighborhood of $(0,0,\textbf{0})$. In fact, in \cite{Monomial summ} it is proved that if $\frac{\d F}{\d \textbf{y}}(0,0,\textbf{0})$ is invertible then the unique formal solution of this equation is $1-$summable in $x^p\varepsilon^q$. These equations were also studied in \cite{BalserMozo} under the same assumptions when $q=1$ and $\textbf{F}$ depends linearly of $\textbf{y}$ considering summability in $x$, in $\e$ and what the authors refer to $(s_1,s_2)-$summability, $ps_1+s_2=1$. Using Theorem \ref{Monomial summability t Borel Laplace} the relation between monomial summability and this $(s_1,s_2)-$summability is clarified and thus the results in \cite{BalserMozo} are valid for general $q$ and $\textbf{F}$.

Taking into account the Borel-Laplace methods introduced in the previous section and in view of formula (\ref{Borel s1 s2 k y derivadas}) we can study under the light of monomial summability formal solutions of the family of partial differential equations \begin{equation}\label{EP principal s1 s2}
x_1^px_2^q\left(\frac{s}{p}x_1\frac{\d \textbf{y}}{\d x_1}+\frac{1-s}{q}x_2\frac{\d \textbf{y}}{\d x_2}\right)=C(x_1,x_2)\textbf{y}+\gamma(x_1,x_2),
\end{equation} where $p,q\in\N_{>0}$, $0<s<1$ and $C$ (resp. $\gamma$) is a square matrix of size $l$ (resp. vector of size $l$) with entries holomorphic functions in some polydisc $D_r^2$ at the origin in $\C^2$. Note that this equation reduces to (\ref{SPDE}) when $s=0$ or $s=1$ and $\textbf{F}(x_1,x_2,\textbf{y})=C(x_1,x_2)\textbf{y}+\gamma(x_1,x_2)$. This will be the only case considered here.

The existence, uniqueness and Gevrey growth of formal solutions of (\ref{EP principal s1 s2}) in contained in the next proposition.

\begin{prop}\label{EP1 has 1 pq gevrey solution} If $C(0,0)$ is invertible then equation (\ref{EP principal s1 s2}) has a unique solution $\normalfont\hat{\textbf{y}}$$=\sum y_{n,m}x_1^nx_2^m$. Moreover $\normalfont \hat{\textbf{y}}$ $\in \left(\C[[x_1,x_2]]^{(p,q)}_1\right)^l$.
\end{prop}

\begin{proof}The existence and uniqueness of $\hat{\textbf{y}}$ follows by expanding $C$ and $\gamma$ in power series at the origin, writing down the corresponding equations in each power and solving recursively the coefficients.

For the Gevrey growth of the solution, for instance for the variable $x_1$, if we write $\hat{\textbf{y}}=\sum y_{n}(x_2)x_1^n$, $C=\sum C_{n,m}x_1^nx_2^m=\sum C_{n}(x_2) x_1^n$ and $\gamma=\sum\gamma_{n,m}x_1^nx_2^m=\sum \gamma_{n}(x_2) x_1^n$ replacing these expressions in the equation we see that
\begin{equation}\label{EP principal en x s1 s2}
\frac{s}{p}(n-p)x_2^q y_{n-p}(x_2)+\frac{1-s}{q}x_2^{q+1}y'_{n-p}(x_2)=\sum_{i=0}^n C_{n-i}(x_2) y_{i}(x_2)+\gamma_{n}(x_2),\hspace{0.2cm} n\in\N.
\end{equation}

\noindent Note that $C_{0}(x_2)$ is also invertible for $|x_2|< r$ (reducing $r$ if necessary). The coefficients $y_{n}$ are uniquely determined by these equations and are holomorphic on $D_r$. Following the same lines as in \cite[Sec. 2]{BalserMozo} with the aid of Nagumo norms (see \cite{CDRSS}) we can conclude that $\hat{\textbf{y}}\in (\C[[x_1,x_2]]_{(1/p,0)})^l$. The same argument for the variable $x_2$ let us conclude that $(\C[[x_1,x_2]]_{(0,1/q)})^l$ and the result follows by equality (\ref{segment of line}) in Section \ref{Sumabilidad}.
\end{proof}

\begin{teor}\label{EP s1 s2 1 summable in xpeq}The unique formal solution $\normalfont\hat{\textbf{y}}$ given by the previous proposition is $1-$summable in $x_1^px_2^q$. The possible singular directions are the ones passing through the eigenvalues of $C(0,0)$.
\end{teor}

\begin{proof}By Theorem \ref{Monomial summability t Borel Laplace} it is enough to show that $\hat{\textbf{y}}$ is $1-(s,1-s)-$Borel summable in the monomial $x_1^px_2^q$. To simplify notations we will write $\mathcal{B}=\mathcal{B}_{\boldsymbol{\alpha}}$, $\hat{\mathcal{B}}=\hat{\mathcal{B}}_{\boldsymbol{\alpha}}$ and $\ast=\ast_{\boldsymbol{\alpha}}$, where $\boldsymbol{\a}=\left(\frac{s}{pk},\frac{1-s}{qk}\right)$. We divide the proof in three steps.

\textit{Step 1. Setting the associated convolution equation.} If $\textbf{w}=x_1^px_2^q \textbf{y}$, equation (\ref{EP principal s1 s2}) is transformed into equation \begin{equation}\label{EP principal s1 s2 2}
x_1^px_2^q\left(\frac{s}{p}x_1\frac{\d \textbf{w}}{\d x_1}+\frac{1-s}{q}x_2\frac{\d \textbf{w}}{\d x_2}\right)=\left(x_1^px_2^qI+C(x_1,x_2)\right)\textbf{w}+x_1^px_2^q\gamma(x_1,x_2),
\end{equation} which is solved formally by $\hat{\textbf{w}}=x_1^px_2^q\hat{\textbf{y}}$. Here $I$ denotes the identity matrix of size $l$. If we apply $\mathcal{B}$ to this equation, using formulas (\ref{Borel s1 s2 k y derivadas}) and (\ref{conv. product}) we see that $\hat{\mathcal{B}}(\hat{\textbf{w}})$, holomorphic at $(0,0)$, is a solution of the convolution equation \begin{equation}\label{EP convolution w}
(\xi_1^p\xi_2^qI-C(0,0))\textbf{F}(\xi_1,\xi_2)=\mathcal{B}(\widetilde{C})\ast \textbf{F}(\xi_1,\xi_2)+g(\xi_1,\xi_2),
\end{equation}

\noindent where $\widetilde{C}(x_1,x_2)=x_1^px_2^qI+C(x_1,x_2)-C(0,0)$ and $g=\mathcal{B}(x_1^px_2^q\gamma)$. Note that \begin{equation}\label{EP g y B_p(A)}
g=\sum\frac{\gamma_{n,m}}{\Gamma\left(1+\frac{ns}{p}+\frac{m(1-s)}{q}\right)}\xi_1^{n}\xi_2^m,\hspace{0.2cm}\mathcal{B}(\widetilde{C})=\sum_{(n,m)\neq (0,0)} \frac{\overline{C}_{n,m}}{\Gamma\left(\frac{ns}{p}+\frac{m(1-s)}{q}\right)}\xi_1^{n-p}\xi_2^{m-q},
\end{equation} where $\overline{C}_{n,m}=C_{n,m}$ for $(n,m)\neq (p,q)$ and $\overline{C}_{p,q}=C_{p,q}+I$.

Let $\lambda_1,...,\lambda_l$ be the eigenvalues of $C(0,0)$ repeated according to their multiplicity. They are all non-zero by assumption. Working on $\Omega=\{(\xi_1,\xi_2)\in \C^2 \hspace{0.1cm}|\hspace{0.1cm} \xi_1^p\xi_2^q\neq\lambda_j\text{ for all }j=1,...,l\}$, the open set where the matrix $\xi_1^p\xi_2^qI-C(0,0)$ is invertible, we see that finding solutions of certain type of the convolution equation (\ref{EP convolution w}) is equivalent to finding a fixed point of the operator $\mathcal{H}$ given by \begin{equation}\label{EP operator}
\mathcal{H}(\textbf{F})(\xi_1,\xi_2)=\left(\xi_1^p\xi_2^qI-C(0,0)\right)^{-1}\left(\mathcal{B}(\widetilde{C})\ast\textbf{F}(\xi_1,\xi_2)+g(\xi_1,\xi_2)\right),
\end{equation} and defined in an adequate Banach space.

\textit{Step 2. Establishing the analytic continuation of $\hat{\mathcal{B}}(\hat{\textbf{w}})$.} Let $U$ be a bounded open such that $\overline{U}\subset \Omega$ and $N\in\N$ to be chosen. We prove that $\hat{\mathcal{B}}(\hat{\textbf{w}})-\sum_{n\leq N}\sum_{m\geq0} \frac{y_{n,m}}{\Gamma(1+ns/p+m(1-s)/q)}\xi_1^n\xi_2^m$ admits analytic continuation on $U$. If we replace $\textbf{w}_{N}=\textbf{w}-\sum_{n\leq N} x_2^{q}y_{n}(x_2)x_1^{n+p}$ in equation (\ref{EP principal s1 s2 2}), then using the recurrences (\ref{EP principal en x s1 s2}) we see that $\textbf{w}_{N}$ satisfies the same differential equation (\ref{EP principal s1 s2 2}) but with $\gamma$ replaced by a $\gamma_{N}$ with $\text{ord}_{x_1} \gamma_N>N$. Therefore $\mathcal{B}(\textbf{w}_N)$ satisfies the same convolution equation (\ref{EP convolution w}) but with $g$ replaced by $g_N=\mathcal{B}(x_1^px_2^q \gamma_N)$ and $\text{ord}_{\xi_1} g_N>N$.

Let $E_{U,N}$ denote the subspace of $\mathbf{F}\in\mathcal{C}(\overline{U},\C^l)\cap \mathcal{O}(U,\C^l)$ such that $$\|\textbf{F}\|_N=\sup_{(\zeta_1,\zeta_2)\in \overline{U}} \frac{|\textbf{F}(\zeta_1,\zeta_2)|}{|\zeta_1|^N},$$ is finite. $E_{U,N}$ is a Banach space with the norm $\|\cdot\|_N$ and $g_N\in E_{U,N}$. We shall prove that $\mathcal{H}_N:E_{U,N}\rightarrow E_{U,N}$, defined as $\mathcal{H}$ but with $g_N$ instead of $g$, is well-defined and it is a contraction if $N$ is large enough. Indeed, if $\textbf{F}\in E_{U,N}$ then \begin{align*}
|\mathcal{H}_N(\textbf{F})&(\xi_1,\xi_2)|\\
&\leq M_U\left|\int_0^1 \mathcal{B}(\widetilde{C})(\xi_1 t^{\frac{s}{p}},\xi_2 t^{\frac{1-s}{q}})\textbf{F}(\xi_1(1-t)^{\frac{s}{p}},\xi_2 (1-t)^{\frac{1-s}{q}})\xi_1^p\xi_2^qdt\right|+M_U\|g_N\|_N|\xi_1|^N\\
&\leq M_U\int_0^1 \left|\mathcal{B}(\widetilde{C})(\xi_1 t^{\frac{s}{p}},\xi_2 t^{\frac{1-s}{q}})\right|\|\textbf{F}\|_N|\xi_1|^{p+N}|\xi_2|^q(1-t)^{Ns/p}dt+M_U\|g_N\|_N|\xi_1|^N,
\end{align*} where $M_U$ is a positive constant such that $\left\|\left(\xi_1^p\xi_2^qI-C(0,0)\right)^{-1}\right\|\leq M_U$, on $\overline{U}$. Using Lemma \ref{bound the integral} below we see that $$|\mathcal{H}_N(\textbf{F})(\xi_1,\xi_2)|\leq M_U\left(\frac{DK_U}{N^a}\|\textbf{F}\|_N+\|g_N\|_N\right)|\xi_1|^N,$$ where $a=\min\left\{\frac{s}{p}, \frac{1-s}{q}\right\}$ and $$K_U=\sup_{(\zeta_1,\zeta_2)\in \overline{U}} \hspace{0.1cm}\sum_{(n,m)\neq (0,0)} \frac{|\overline{C}_{n,m}|}{\Gamma\left(\frac{ns}{p}+\frac{m(1-s)}{q}\right)}|\zeta_1|^n|\zeta_2|^m,$$ is finite since $U$ is bounded. Then $\mathcal{H}_N(\textbf{F})\in E_{U,N}$ and furthermore, if $\textbf{F},\textbf{G}\in E_{U,N}$ then $$\|\mathcal{H}_N(\textbf{F})-\mathcal{H}_N(\textbf{G})\|_N\leq \frac{DM_UK_U}{N^a}\|\textbf{F}-\textbf{G}\|_N.$$ Thus $\mathcal{H}_N$ is a contraction if $DM_UK_U<N^a$. By Banach's fixed point theorem $\mathcal{H}_N$ has a unique fixed point $\textbf{F}_{U,N}\in E_{U,N}$, i.e., equation (\ref{EP convolution w}) has a unique holomorphic solution defined on $U$ of the form $\textbf{F}_U=\textbf{F}_{U,N}+\sum_{n\leq N}\sum_{m\geq 0}\frac{y_{n,m}}{\Gamma(1+ns/p+m(1-s)/q)}\xi_1^n\xi_2^m$.

Since the unique solution of equation (\ref{EP convolution w}) in a small polydisc at $(0,0)$ contained in $\Omega$ is $\hat{\mathcal{B}}(\hat{w})$, if $U$ intersects this polydisc, $\textbf{F}_U$ and $\hat{\mathcal{B}}(\hat{\textbf{w}})$ coincide in that intersection. We can conclude that $\hat{\mathcal{B}}(\hat{\textbf{w}})$ admits analytic continuation to $\Omega$.

\textit{Step 3. Determining the exponential growth of solutions of the convolution equation.} It remains to prove that the above solutions have exponential growth of order at most $\left(\frac{pk}{s},\frac{qk}{1-s}\right)$. Let $K$ be a positive constant and $S$ be an unbounded open set such that $\overline{S}\subset\Omega$. We will denote by $E_{S,K}$ the subspace of $\mathbf{F}\in\mathcal{O}(S,\C^l)$ such that $$\|\textbf{F}\|_{K}=\sup_{(\zeta_1,\zeta_2)\in S} |\textbf{F}(\zeta_1,\zeta_2)|e^{-KR(\zeta_1,\zeta_2)},\hspace{0.3cm} R(\zeta_1,\zeta_2)=\max\{|\zeta_1|^\frac{p}{s},|\zeta_2|^{\frac{q}{1-s}}\},$$ is finite. Then $E_{S,K}$ is a Banach space with the norm $\|\cdot\|_K$ and from (\ref{EP g y B_p(A)}) we see we can find a large enough constant $K'>0$ such that $g\in E_{S,K'}$ for any such $S$.

As before we prove that $\mathcal{H}:E_{S,K}\rightarrow E_{S,K}$, is well-defined and a contraction if $K>K'$ is large enough. If $\textbf{F}\in E_{S,K}$, then \begin{align*}
|\mathcal{H}&(\textbf{F})(\xi_1,\xi_2)|\\
&\leq M_S\left|\int_0^1 \mathcal{B}(\widetilde{C})(\xi_1 t^{\frac{s}{p}},\xi_2 t^{\frac{1-s}{q}})\textbf{F}(\xi_1(1-t)^{\frac{s}{p}},\xi_2 (1-t)^{\frac{1-s}{q}})\xi_1^p\xi_2^qdt\right|+M_S\|g\|_{K'}e^{K'R(\xi_1,\xi_2)}\\
&\leq M_S\int_0^1 \left|\xi_1^p\xi_2^q \mathcal{B}(\widetilde{C})(\xi_1 t^{\frac{s}{p}},\xi_2 t^{\frac{1-s}{q}})\right|\|\textbf{F}\|_Ke^{KR(\xi_1,\xi_2)(1-t)}dt+M_S\|g\|_{K'}e^{K'R(\xi_1,\xi_2)},
\end{align*} where $M_S$ is a constant such that $\left\|\left(\xi_1^p\xi_2^qI-C(0,0)\right)^{-1}\right\|\leq M_S$, on $\overline{S}$. To bound this term we use the following estimate \begin{align*}
\int_0^1 t^{\frac{ns}{p}+\frac{m(1-s)}{q}-1}e^{KR(\xi_1,\xi_2)(1-t)}dt\leq \frac{e^{KR(\xi_1,\xi_2)}\Gamma\left(\frac{ns}{p}+\frac{m(1-s)}{q}\right)}{(KR(\xi_1,\xi_2))^{\frac{ns}{p}+\frac{m(1-s)}{q}}},
\end{align*} valid for all $n,m\in \N$, $(n,m)\neq (0,0)$ and $\xi_1,\xi_2\in\C^*$. Then we have \begin{align*}
\left|\mathcal{H}(\textbf{F})(\xi_1,\xi_2)\right|&\leq M_SL\left(\frac{1}{K^{\frac{s}{p}}}+\frac{1}{K^{\frac{1-s}{q}}}\right)\|\textbf{F}\|_Ke^{KR(\xi_1,\xi_2)}+M_S\|g\|_{K'}e^{K'R(\xi_1,\xi_2)}\\
&\leq M_S\left(\frac{2L}{K^a}\|\textbf{F}\|_K+\|g\|_{K'}\right)e^{KR(\xi_1,\xi_2)},
\end{align*} where $a=\min\{\frac{s}{p}, \frac{1-s}{q}\}$ and $L$ is a constant such that $$\sum_{m\geq 1}|\overline{C}_{0m}||\zeta_2|^{m-1}, \sum_{{n\geq 1}\atop {m\geq 0}}|\overline{C}_{nm}||\zeta_1|^{n-1}|\zeta_2|^{m}\leq L,\hspace{0.2cm}\text{ on } D_{r'}^2,$$ $r'<r$ and $K$ chosen such that $1/K^a<r'$. Therefore $\mathcal{H}(\textbf{F})\in E_{S,K}$. In the same way, if $\textbf{F},\textbf{G}\in E_{S,K}$ then $$\|\mathcal{H}(\textbf{F})-\mathcal{H}(\textbf{G})\|_K\leq \frac{2M_SL}{K^a}\|\textbf{F}-\textbf{G}\|_K,$$ and
if we take $K$ such that $2M_SL<K^a$ we can conclude that $\mathcal{H}$ is a contraction and it has a unique fixed point. This means that (\ref{EP convolution w}) has a unique solution defined in $S$ with the exponential growth as above.

If we choose any direction $d\neq \text{arg}(\lambda_j)$, $j=1,...,l$ and $\theta_d>0$ small enough such that $S=S_{p,q}(d, 2\theta_d)\subset \Omega$, then we have proved in particular that $\hat{\mathcal{B}}(\hat{w})$ can be analytically continued to $S$ with exponential growth as required in Definition \ref{def k-s_1,s_2-Borel sum} for $k=1$. Then $\hat{\textbf{y}}$ es $1-(s,1-s)-$Borel summable in $x_1^px_2^q$ in every direction $d\neq \text{arg}(\lambda_j)$ as we wanted to prove.
\end{proof}

\begin{lema}\label{bound the integral} Let $p,q,N\in\N_{>0}$ be given and $n,m\in\N$ not both zero. For $0<s<1$ there is a constant $D$ independent of $n,m$ such that $$\int_0^1 t^{\frac{ns}{p}+\frac{m(1-s)}{q}-1}(1-t)^{\frac{Ns}{p}}dt\leq \frac{D}{N^a},\hspace{0.3cm}  a=\min\left\{\frac{s}{p}, \frac{1-s}{q}\right\}.$$
\end{lema}

\begin{proof}Let $J$ denote the integral. Using the Beta function we can write
$$J=\frac{\Gamma\left(\frac{ns}{p}+\frac{m(1-s)}{q}\right)\Gamma\left(1+\frac{Ns}{p}\right)}{\Gamma\left(1+\frac{(n+N)s}{p}+\frac{m(1-s)}{q}\right)}.$$ To prove the lemma we consider two situations. First, if $\frac{ns}{p}+\frac{m(1-s)}{q}>1$, we use the inequality (\ref{inq Gamma}) to bound $J$ by $\frac{p/s}{N}\leq \frac{p/s}{N^a}$. In the finite remaining cases we can use the limit $\lim_{N\rightarrow+\infty}\frac{(N\sigma+b)^{b}\Gamma(1+N\sigma)}{\Gamma(1+N\sigma+b)}=1$, with $\sigma=\frac{s}{p}$ and $b=\frac{ns}{p}+\frac{m(1-s)}{q}\geq a$ to find positive constants $D_{n,m}^{p,q,s}$ such that $J$ is bounded by $\frac{D_{n,m}^{p,q,s}}{\left(\frac{Ns}{p}+b\right)^b}\leq \frac{(p/s)^bD_{n,m}^{p,q,s}}{N^a}$. The result follows taking $D$ large enough.
\end{proof}

To finish this section we want to apply Theorem \ref{EP s1 s2 1 summable in xpeq} to exhibit a case when certain Pfaffian systems with normal crossings, in case of having a formal solution, this solution converges. We will be concerned with systems of the form \begin{subnumcases}\empty\label{EP principal x}
x_2^qx_1^{p+1}\frac{\d \textbf{y}}{\d x_1}=A(x_1,x_2)\textbf{y}+\gamma_1(x_1,x_2),\\
\label{EP principal e} x_1^{p}x_2^{q+1}\frac{\d \textbf{y}}{\d x_2}=B(x_1,x_2)\textbf{y}+\gamma_2(x_1,x_2),
\end{subnumcases} where $p,q\in\N_{>0}$ and $A,B$ (resp. $\gamma_1, \gamma_2$) are square matrixes of size $l$ (resp. vectors of size $l$) with entries holomorphic functions in some polydisc $D_r^2$ at the origin in $\C^2$. The same problem was studied in \cite{CM} for systems where the monomials involved in each equation where different and conditions to obtain convergence of a formal solution were given using tauberian properties for monomial summability.

A natural condition to request for these systems and to relate the solutions of both equations is the \textit{integrability condition} that in this context can be written as the pair of equations \begin{align}\label{EP integrability AB}
x_1^{p}x_2^{q}&\left(x_2\frac{\d A}{\d x_2}-qA\right)-x_1^{p}x_2^{q}\left(x_1\frac{\d B}{\d x_1}-pB\right)+[A,B]=\mathbf{0},\\
x_1^{p}x_2^{q}&\left(x_2\frac{\d \gamma_1}{\d x_2}-q\gamma_1\right)-x_1^{p}x_2^{q}\left(x_1\frac{\d \gamma_2}{\d x_1}-p\gamma_1\right)+A\gamma_2-B\gamma_1=\mathbf{0}, \hspace{0.2cm}\text{ on } D_r^2.
\end{align} This is a restrictive condition, for instance if (\ref{EP integrability AB}) holds then for every eigenvalue $\mu$ of $B(0,0)$ there is an eigenvalue $\lambda$ of $A(0,0)$ such that $q\lambda=p\mu$. The number $\lambda$ is an eigenvalue of $A(0,0)$, when restricted to its invariant subspace $\{v\in\C^l | (B(0,0))-\mu I)^kv=0\text{ for some }k\in\N\}$ (see \cite[Thm. 4.4]{CM}). In particular this implies that the spectrum of $\frac{s}{p}A(0,0)+\frac{1-s}{q}B(0,0)$, $0\leq s\leq 1$ is independent of $s$.

If we multiply (\ref{EP principal x}) by $\frac{s}{p}$, (\ref{EP principal e}) by $\frac{1-s}{q}$ and add them we get an equation of the form (\ref{EP principal s1 s2}) where $C_s=\frac{s}{p}A+\frac{1-s}{q}B$ and $\gamma_s=\frac{s}{p}\gamma_1+\frac{1-s}{q}\gamma_2$. Using Theorem \ref{EP s1 s2 1 summable in xpeq} we have the following proposition describing some summability properties of solutions of such systems.

\begin{prop}\label{EP has 1 pq gevrey solution}The following assertions hold:
\begin{enumerate}
\item If the system (\ref{EP principal x}), (\ref{EP principal e}) is completely integrable and $A(0,0)$ or $B(0,0)$ is invertible then the system (\ref{EP principal x}), (\ref{EP principal e}) has a unique formal solution. This solution is $1-$summable in $x_1^px_2^q$.
\item If the system has a formal solution $\normalfont\hat{\textbf{y}}$ and $C_s(0,0)$ is invertible for some $0<s<1$, then $\normalfont\hat{\textbf{y}}$ is $1-$summable in $x_1^px_2^q$. The possible singular directions are the those passing through the eigenvalues of $C_s(0,0)$.
\end{enumerate}
\end{prop}

In the non-integrable case there are not imposed relations between $A(0,0)$ and $B(0,0)$ and then there are more possible situations for the spectrum of $C_s(0,0)$. In particular there is a case when we can conclude convergence due to the absence of singular directions as explained below.

\begin{teor}\label{ultimo caso pfaffianos} Assume that the system (\ref{EP principal x}), (\ref{EP principal e}) has a formal solution $\normalfont\hat{\textbf{y}}$. Let $\lambda_1(s),...,\lambda_l(s)$ denote the eigenvalues of $\frac{s}{p}A(0,0)+\frac{(1-s)}{q}B(0,0)$, $0\leq s\leq 1$, and assume that they are never zero. If for every direction $d$ there is $s\in[0,1]$ such that $\text{arg}(\lambda_j(s))\neq d$ for all $j=1,...,l$ then $\normalfont\hat{\textbf{y}}$ is convergent.
\end{teor}

\begin{proof}Using Proposition \ref{EP has 1 pq gevrey solution} (2) we see that $\hat{\textbf{y}}$ has no singular directions for $1-$summability in $x_1^px_2^q$ and thus it is convergent.
\end{proof}

\section{Further developments} \label{further}

In order to treat more general equations, it will be necessary to introduce a notion of multisummability, with respect to several monomials, for instance, as in the one variable case.
A possible approach would be through generalized acceleration operators, following the ideas of \'{E}calle in the theory of one variable.
Indeed, an analogue to accelerator operators can be defined by formally computing the composition between a Borel and Laplace transforms of different indexes. If $p,q,p',q'\in\N_{>0}$, $k,k'>0$ and $0<s,s'<1$ are given it can be shown that \begin{align*}
\mathcal{B}_{\boldsymbol{\beta}}\circ\mathcal{L}_{\boldsymbol{\a}}(f)(\xi_1,\xi_2)=\frac{(\xi_1^p\xi_2^q)^k}{(\xi_1^{p'}\xi_2^{q'})^{k'}}\int_0^{e^{i\theta}\infty} f(\xi_1\tau^{\frac{s}{pk}},\xi_2\tau^{\frac{1-s}{qk}}C_{\Lambda k'/k}(\tau)d\tau,
\end{align*} where $C_{\Lambda k'/k}$ is an \textit{acceleration function} of Ecalle, $\boldsymbol{\alpha}=\left(\frac{s}{pk},\frac{1-s}{qk}\right)$, $\boldsymbol{\beta}=\left(\frac{s'}{p'k'},\frac{1-s'}{q'k'}\right)$, $|\theta|<\pi/2$ and it is requested that $\Lambda:=\frac{s}{s'}\frac{p'}{p}=\frac{1-s}{1-s'}\frac{q'}{q}>k/k'$. We refer the reader to \cite{CThesis} for more information.

Further development of these ideas, and applications, will be the content of a future work.

\bibliographystyle{plaindin_esp}

\end{document}